\documentclass{article}
\usepackage{amsfonts,amsthm}
\usepackage{cancel}
\usepackage{enumerate,young,amsmath,graphicx,float,young,color,textcomp}
\newtheorem{theorem}{Theorem}[section]
\newtheorem{corollary}[theorem]{Corollary}
\newtheorem{lemma}[theorem]{Lemma}

\newtheorem{proposition}[theorem]{Proposition}

\newtheorem{reduction}{Reduction}

\newcommand{\myqed}{\hfill \ensuremath{\Box} \newline}

\newcommand{\ignore}[1]{}

\begin{document}

\title{On the Lipschitz constant of the RSK correspondence} 

\author{Nayantara Bhatnagar \thanks{Department of Computer Science,
    Hebrew University of Jerusalem, {\tt nayantara@cs.huji.ac.il}.
  Supported by a Lady Davis Fellowship.} \and
  Nathan Linial \thanks{Department of Computer Science, Hebrew
    University of Jerusalem, {\tt nati@cs.huji.ac.il}.}}

\maketitle 

\begin{abstract}
We view the RSK correspondence as associating to each permutation
$\pi \in S_n$ a Young diagram $\lambda=\lambda(\pi)$, i.e.
a partition of $n$. Suppose now that $\pi$ is left-multiplied by
$t$ transpositions,
what is the largest number of cells in $\lambda$
that can change as a result? It is natural refer to this question
as the search for the Lipschitz constant of the RSK correspondence.

We show upper bounds on this Lipschitz constant as a function of $t$.
For $t=1$, we give a construction of permutations that achieve this
bound exactly. For larger $t$ we construct permutations which come
close to matching the upper bound that we prove.

\end{abstract}

\section{Introduction}
The Robinson-Schensted-Knuth (RSK) correspondence \cite{Rob,Sch,Knu}
maps an arbitrary 
permutation $\pi \in S_n$ 
bijectively to an ordered
pair of Young tableaux of the same shape $\lambda=\lambda(\pi)$.
How much can $\lambda(\pi)$ change as we mildly vary $\pi$?
For example, if we left-multiply $\pi$ by $t$
transpositions, to what extent can $\lambda$ change\footnote{We use of
  the following standard asymptotic notation. We say $f(n) = o(g(n)$ iff $
  \lim_{n \to \infty} f(n)/g(n)=0$ and $f(n)=O(g(n))$ iff there is a
  constant $C$ and $n_0$ s.t. for $n>n_0$, $f(n) \le C
  g(n)$. Finally, $f(n) = \Omega(g(n))$ iff there is a
  constant $C$ and $n_0$ s.t. for $n>n_0$, $f(n) \ge C
  g(n)$}? We
begin with the case when $t=1$ and show
that the resulting Young diagram can differ from $\lambda$ on
at most $\sqrt{n/2}$ cells. We show 
that this bound is tight by giving explicit constructions of permutations $\pi$
for which this bound is attained where the diagrams differ in at least
$(1-o(1))\sqrt{n/2}$ cells.
We then turn to consider the same question for larger $t$ and show that
the corresponding diagram changes in at most
$O(\sqrt{nt \ln t})$
cells. The best constructions we know nearly
match this bound and yield, e.g., $(1-o(1))\sqrt{nt/2}$ changes
for $t=o(n)$.  
 
The outline of this paper is as follows. In the remainder of
this section we
recall some definitions and properties of Young tableaux and the RSK
correspondence. 
In Section
\ref{sec:t=1} we prove upper bounds on the Lipschitz
constant when $t=1$ and show a matching construction. In
Section \ref{sec:upper-bounds-general-t} we give upper bounds and
extend our constructions for the
case of general $t$. We conclude with some directions for further
research in Section \ref{sec:conclusions}.

\subsection{Notation and Preliminaries}
We recall some definitions and background on Young Tableaux and the RSK
algorithm here. For more detailed expositions refer to \cite{Ful, Mac}
or \cite{Sta}.

%\begin{definition}
Let $n \in \mathbb N$ be a positive integer. A vector
$\lambda=(\lambda_1, \lambda_2, \dots)$ of positive integers  is a
{\em partition} of $n$ (denoted by $\lambda \vdash n$)
if
\[
\lambda_1 \ge \lambda_2 \ge \dots > 0 \ \mathrm{and}
\ \sum_{i}\lambda_i = n.
\]
%\end{definition}

%\begin{definition}
The {\em Young diagram} (or diagram) of a partition $\lambda$ is a left-justified
array of cells with $\lambda_i$ cells in the $i$-th row for each
$i\geq 1$.
%\end{definition}
For example, the diagram of the partition $(5,5,3,2)$ is

\[
\begin{Young}
& & & &\cr
& & & &\cr
& &  \cr
& \cr
\end{Young}
\]

The cell in the $i$-th row and $j$-th column is referenced by its {\em
coordinate} $(i,j)$. Thus $(1,1)$ is the top leftmost cell of the diagram.

The {\em conjugate} of a partition $\lambda$, denoted by $\lambda'$ is
the partition whose diagram is the transpose of the diagram of $\lambda$.

A {\em standard Young tableau} (SYT or tableau) of size $n$ with
entries from $[n]$ is a
diagram whose cells are filled with the elements of $[n]$ in such a
way that the entries are strictly increasing from left to right along a
row as well as from top to bottom down a column. The {\em shape} of a
tableau $T$, denoted $sh(T)$ is the partition corresponding to the
diagram of $T$. For example, 

\[
\begin{Young}
1& 2& 4 & 7 \cr
3& 6 \cr
5   \cr
\end{Young}
\]

is a tableau of size $7$ of shape $(4,2,1)$. Note that the elements in
the cells of a SYT are distinct integers. Let
$\mathcal T_n$ denote the set of SYT of size $n$.

\subsection{The Robinson-Schensted-Knuth (RSK) Correspondence}
The RSK correspondence discovered by Robinson \cite{Rob}, Schensted
\cite{Sch} and further extended by Knuth \cite{Knu} is a bijection
between the set of permutations $S_n$ and pairs of tableau of size $n$
of the same shape. This bijection is intimately related to the 
representation theory of the symmetric group \cite{Jam,Dia}, the theory of 
symmetric functions \cite[Chapter 7]{Sta}, and the theory of
partitions \cite{And}.

The bijection can be defined through a {\em row-insertion} algorithm
first defined by Schensted \cite{Sch} in order to study the longest
increasing subsequence of a permutation. Suppose that we have a tableau
$T$. The row-insertion procedure below inserts a positive integer $x$
that is distinct from all entries of $T$, into
$T$ and results in a tableau denoted by $T \leftarrow x$.
\begin{enumerate}
\item Let $y$ smallest number larger than $x$ in the first row of
  $T$. Replace the cell containing $y$ with $x$. If there is no such
  $y$, add a cell containing $x$ to the end of the row.
\item If $y$ was removed from the first row, attempt to insert it into
  the next row by the same procedure as above. If there is no row to
  add $y$ to, create a new row with a cell containing $y$.
\item Repeat this procedure on successive rows until either a number
  is added to the end of a row or added in a new row at the bottom.
\end{enumerate}

The RSK correspondence from $S_n$ to $\{(P,Q) \in \mathcal T_n
\times \mathcal T_n\ : \ sh(P)=sh(Q)\}$ can now be defined as
follows. Let $\pi \in 
S_n$ and let $\pi_i$ denote the element of $[n]$ in position $i$ in
$\pi$. Let $P_1$ be the tableau with a single cell containing
$\pi_1$. Let $P_{j} = P_{j-1} \leftarrow \pi_j$ for all $1<j \le n$ and
set $P=P_n$. The tableau $Q$ is defined recursively in terms of tableaux $Q_i$
of size $i$ as follows. Let $Q_1$ be
the tableau with one cell containing the integer $1$. The equality of shapes 
$sh(Q_i)=sh(P_i)$ is maintained throughout the process. The cell of $Q_i$ containing $i$ is the
(unique) cell of $P_i$ that does not belong to $P_{i-1}$. The remaining cells of
$Q_i$ are identical to those of $Q_{i-1}$. Finally, set
$Q=Q_n$. We refer to $P$ as the {\em insertion tableau} and $Q$ is
the {\em recording tableau}.

Let $\pi \in S_n$ and let $(P,Q)$ be the corresponding tableaux under
the RSK correspondence. The
{\em shape of $\pi$} is $sh(P)=sh(Q)$ and will be denoted by
$\lambda=\lambda(\pi)$.
The RSK correspondence has numerous interesting properties
(see \cite{Ful,Mac,Knu-book} or \cite{Sta}). Some that will be useful in
particular are as follows.

\begin{proposition}
\label{prop:reverse_conjugate}
Let $\lambda=\lambda(\pi)$. Then
the diagram corresponding to $\pi^R$, the reversal of $\pi$,
is $\lambda'$, the conjugate of $\lambda$.
\end{proposition}

\begin{proposition}
\label{prop:inverse_RSK}
Let $(P,Q)$ be the tableaux corresponding to a permutation $\pi$ under
the RSK correspondence. Then the tableaux 
corresponding to the inverse permutation $\pi^{-1}$ are $(Q,P)$. Thus
the shape remains invariant upon inversion, i.e.,  $\lambda(\pi^{-1})
= \lambda(\pi)$. 
\end{proposition}

\subsection{Motivation and Related Work}

In view of the important role of the RSK correspondence, it is
natural to investigate various aspects of it. Thus
Fomin's appendix in \cite[Chapter 7]{Sta} starts with the following two
motivating questions:
\begin{enumerate}[(1)]
\item Given a partition $\lambda$, characterize those permutations $\pi$
for which $\lambda(\pi)=\lambda$.
\item Given a tableau $P$, characterize the permutations $\pi$ which have
$P$ as their insertion tableau.
\end{enumerate}

We consider an {\em approximate} version of such questions and ask
to what extent $\lambda$ changes as $\pi$ changes slightly.
Question (1) is answered by the following theorem of Greene.

\begin{theorem}[Greene \cite{Gre}]\label{thm:Greene}Let $\pi$ be a permutation,
 and suppose that the largest cardinality of the union of $j$
  increasing subsequences in $\pi$ is $\mu_j$,
then  $\lambda(\pi)=\lambda_1,\dots,\lambda_k$, where
$\lambda_1=\mu_1$ and $\lambda_j = \mu_j - \mu_{j-1}$ for all $j \ge 2$.
\end{theorem}

In his study of the RSK correspondence, Knuth discovered certain
equivalence relations that are key to the solution of Question (2)
above. Two permutations are {\em Knuth 
equivalent} if one can be obtained from the other by certain
restricted sequences of adjacent transpositions. Knuth equivalent
permutations are the 
equivalency classes of permutations that have the same insertion tableau. 
For more on the subject, see~\cite{Sta}.

In order to make our question concrete, we need to specify two
measures of distance: One between permutations and the other between
diagrams. A natural metric on permutations is {\em left-multiplication} by
adjacent transpositions. An {\em 
adjacent transposition} is a permutation of the form
$(i,i+1)$. Left-multiplying $\pi$ by an adjacent transposition is denoted
by $(i,i+1) \circ \pi$ and means that first, the permutation $\pi$ is
applied and then the transposition. 
We denote the least number of adjacent
transpositions that transform the permutation $\pi$ to $\tau$
by $d(\pi, \tau)$. Recall that $d(\cdot,\cdot)$ is the
graph metric in the Cayley graph of $S_n$ w.r.t. the generating set of
adjacent transpositions $(1,2), (2,3), \dots, (n,n-1)$. We will say
that two permutations $\pi$ and $\tau$ are {\em at distance $t$} if
$d(\pi,\tau)=t$. 
If $\lambda$ and $\mu$ are two diagrams, define their
distance to be 
\[
\Delta=\Delta(\lambda,\mu):=\frac{1}{2}\sum_{i=1}^n |\lambda_i - \mu_i|.
\]
Let $\pi$ and $\tau$ be any two permutations. We are interested in the
Lipschitz constant of this mapping, i.e., 
\[
L(n,t):=\max \frac{\Delta(\lambda(\pi),\lambda(\tau))}{d(\pi,\tau)}
\]
where the maximum is over all $\pi, \tau \in S_n$ with
$d(\pi,\tau)=t$. 

The choice of left-multiplication above is in fact without loss of
generality. By Proposition \ref{prop:inverse_RSK} the shape of a
permutation and its inverse under the RSK correspondence are the
same. Our results thus all 
follow immediately for right-multiplication since $\tau=(i,i+1) \circ \pi$ is
equivalent to $\pi^{-1} = \tau^{-1} \circ (i,i+1)$.

In general, although $d$ and $\Delta$ are natural metrics to study for
permutations and diagrams respectively, the same question can be asked
for other metrics. We discuss the extension of our results to other metrics on
permutations in Section~\ref{sec:conclusions}.

%The shape $\lambda$ of the Young diagram under the RSK correspondence
%is well-studied in other contexts. One of the most interesting
%lines of work is on the 
%distribution of the row lengths of $\lambda(\pi)$ for a random
%permutation $\pi$. 
%Okunkov \cite{Oku}, Johansson \cite{Joh} and Borodin, Okounkov,
%and Olshanskii \cite{BOO} showed
%that the limiting law of the first $k$ rows behaves like the first $k$
%largest eigenvalues of the random $n \times n$ traceless GUE
%matrix.

\section{Exact Bounds on the Lipschitz Constant for a Single Transposition.}
\label{sec:t=1}
In this section we will show upper bounds on the Lipschitz constant
when the number of transpositions $t=1$. We also give a construction
of a family of permutations which achieve this bound asymptotically.

\subsection{Upper Bounds}

The first step of the proof is to show that left-multiplying a
permutation by a transposition can result in only a bounded number of
cells being different in each row of the diagram.

\begin{proposition}
Let $\pi,\tau \in S_n$ and let $\lambda, \mu$ be the respective
diagrams. Suppose that $\tau=(i,i+1) \circ \pi$,
and $\pi_i < \pi_{i+1}$. Then,
\begin{align}\label{eq:running-sum-t1}
\forall \ 1 \leq j \leq n,  \ \ \ \  \sum_{i=1}^j \mu_i \leq
\sum_{i=1}^j \lambda_i \leq \sum_{i=1}^j \mu_i +1 
\end{align}
\label{prop:running-sum-t1}
\end{proposition}

\begin{proof}
Suppose that 
the largest cardinality of the union of $j$ increasing subsequences
in $\pi$ is $\ell$. Suppose there is a subsequence which includes the
pair that is being transposed in $\pi$. By
deleting one of the elements of the pair we obtain a set of $j$
increasing subsequences of $\tau$ whose cardinality is at least
$\ell-1$. If there is no such subsequence, then the same $j$ subsequences
are also increasing in $\tau$. By Greene's Theorem \ref{thm:Greene}
this implies
$\sum_{i=1}^j \lambda_i \leq \sum_{i=1}^j \mu_i +1.$

For the lower bound, consider the largest cardinality of the 
union of $j$ increasing sequences in $\tau$. No
subsequence in this union can contain both of the elements involved in the
transposition. Since the pair involved in the
transposition have no other elements between them in both $\pi$ and $\tau$
the subsequences are also increasing in $\pi$. We conclude in the same way that
$\sum_{i=1}^j \mu_i \leq \sum_{i=1}^j \lambda_i.$
\end{proof}

Figure \ref{fig:tableaux-t1} will be useful in the following
discussion. It depicts the union of two diagrams $\lambda$  
and $\mu$, which is also a Young diagram. The symmetric 
difference consists of the cells marked by a dot.
The remaining set of cells of the diagram labeled $W$ is the intersection of
$\lambda$ and $\mu$ and this is a Young diagram as well.

\begin{corollary}\label{cor:one-cell-different} 
Let $\pi, \tau \in S_n$ where $d(\pi,\tau)=1$ and let
$\lambda,\mu$ be the corresponding diagrams. Then at most one cell in
each row and 
each column of the union of $\lambda$ and $\mu$ can be in the
symmetric difference.
\end{corollary}
\begin{proof}
To see this for a row $r$, consider inequality
\eqref{eq:running-sum-t1} for $j=r$ and for $j=r-1$ and
take their difference. A similar argument applied to the reversed
permutations implies the claim for columns. (see
Proposition~\ref{prop:reverse_conjugate}). 
\end{proof}

\begin{theorem}
Let $\pi$ and $\tau$ be permutations in $\mathcal S_n$ with
respective Young diagrams $\lambda$ and $\mu$, and suppose that
$d(\pi,\tau)=1$. Then
\[
\Delta= \Delta(\lambda,\mu) \leq \sqrt{\frac n2}.
\] 
\label{thm:upperbound-t=1}
\end{theorem}

\begin{figure}[t]
\center   
\resizebox{7cm}{!}{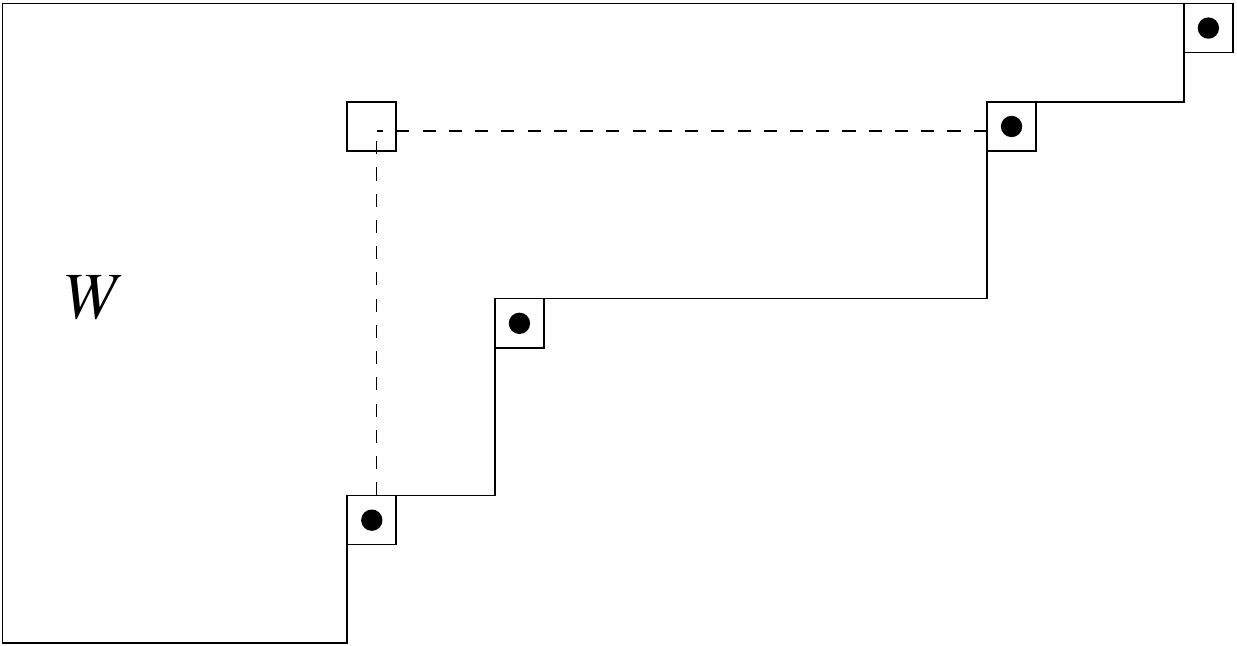}
\caption{The union of $\lambda$ and $\mu$ with cells of the symmetric
  difference marked.} 
\label{fig:tableaux-t1}
\end{figure}

\begin{proof}
As shown in Figure \ref{fig:tableaux-t1},
let $(i,j)$ and $(i',j')$ be the coordinates of two distinct cells in
the symmetric 
difference. By Corollary \ref{cor:one-cell-different}, $i \neq i'$ and
$j \neq j'$, and 
$(\min(i,i'),\min(j.j')) \in W$. This gives a 1:1 map from
unordered pairs of cells in the symmetric difference into $W$.
Therefore,
\[
{2\Delta \choose 2} \leq n-\Delta 
\]
implying the required bound
\[
\Delta \leq \sqrt{\frac n2}. \qedhere
\]
\end{proof}

\subsection{Construction}
In this section we construct pairs of permutations in $S_n$
which differ by a single transposition whose corresponding
Young diagrams differ by at least  $(1-o(1))\sqrt{n/2}$
cells, matching the upper bound in 
Theorem \ref{thm:upperbound-easy} asymptotically. The following lemma
characterizes the shape of a permutation by the cardinalities of increasing
and decreasing subsequences.

\begin{lemma}\label{lem:decreasing-witness}Let $\pi \in S_n$ be
a permutation whose elements can be decomposed in the following two ways:
(i) into increasing
subsequences of cardinalities $\lambda_1,\lambda_2,\dots$, and (ii) into decreasing
subsequences of cardinalities $\lambda_1',\lambda_2',\dots$, where the partitions
$\lambda$ and $\lambda'$ are conjugate. Then $\lambda=\lambda(\pi)$.
\end{lemma}

\begin{proof}
By Greene's Theorem \ref{thm:Greene} it suffices to show that for
each $r$, the largest cardinality of the
union of $r$ increasing sequences in $\pi$ is $\sum_{i \leq r}\lambda_i$.
By assumption we know it is at least this number and we need to show the
opposite inequality. Namely, that if
$s_1,\dots,s_j$ is a collection of disjoint increasing sequences in $\pi$,
$\sum_{i=1}^r|s_i| \le \sum_{i \leq r}\lambda_i$.

By assumption, there is a decomposition $d_1,d_2\dots$ of $\pi$
into disjoint decreasing subsequences
of cardinalities $\lambda_1',\lambda_2',\dots$. But
each $s_i$ and $d_j$ can have at most one element in common, so that
\[
\sum_{i=1}^r|s_i| = \sum_{r \ge i \ge 1,~j} |s_i \cap d_j|
\leq \sum_j \min\{|d_j|,r\} = \sum_j \min\{\lambda_j',r\} = \sum_{i=1}^r \lambda_i
\]
where the last equality follows because the partitions
$\lambda_1,\lambda_2,\dots$ and $\lambda_1',\lambda_2',\dots$ are
conjugate. 
\end{proof}

\begin{theorem}\label{thm:lower-bound-one-transp}
For every $n$ there are permutations $\pi,\tau \in S_n$ with
$d(\pi,\tau)=1$ and respective shapes $\lambda,\mu$ 
such that $\Delta(\lambda,\mu) \geq (1-o(1))\sqrt{n/2}$. 
\end{theorem}
%Let $n=(k+1)^2/2$ for $k \ge 1$ where $k$ is odd. Then, their exists a pair of
%permutations $\pi_1,\pi_2 \in \mathcal S_n$ differing by exactly $1$
%transposition such that the shapes of their corresponding tableaux
%differ by exactly 1 in each row. In particular the permutations have
%shapes $\lambda=k+1,k-1,k-1,\dots,2,2$ and $\mu=
%k,k,k-2,k-2,\dots,1,1$ and $\Delta=\sqrt{n/2}$. 

\begin{proof}
Our proof says, in fact, a little more than what is stated.
Namely for $n=(k+1)^2/2$ with
$k$ an odd integer, we will construct two permutations
$\pi$ and $\tau$ of shapes
$\lambda=(k+1,k-1,k-1,\dots,2,2)$ and $\mu=(k,k,k-2,k-2,\dots,1,1)$
which differ by exactly one cell in each row and column, giving
$\Delta=\sqrt{n/2}$. 
Thus it can be verified that together with Theorem~\ref{thm:upperbound-t=1}
this gives a complete answer to our question for $n$
of this form. For other values of $n$ we get the result by
padding this basic construction.
In the discussion that follows we decompose these permutations
into monotone subsequences. The decompositions we exhibit are not
necessarily unique, but for our purpose any decomposition suffices. 

The construction can, perhaps, be best understood by observing
alongside with the general discussion
a concrete special case. So we intersperse our general constructions
with an illustration that shows how things work
for $n=18$ ($k=5$). We start by dividing
the elements of $[n]$ into three
categories according to their magnitude. The ``small''
elements are those in the interval $[1,n/2-\frac{k+1}{2}]$. The next $k+1$
elements, i.e., interval $[n/2-\frac{k-1}{2},n/2+\frac{k+1}{2}]$ are
``intermediate'' and members of the interval
$[n/2+\frac{k+3}{2},n]$ are ``big''.

We further subdivide the big elements (in order) into blocks
$b_1,\dots,b_{(k-1)/2}$. 
The small elements are split (in order) into blocks $s_{(k-1)/2}, \ldots,s_1$.
Both $s_i$ and $b_i$ have cardinality $2i$.

\begin{align*}
\begin{array}{ccc}
s_2 \ \ \ \ \ \ \ \ s_1 & \ \ \ \ \ \ \ \ & b_1
\ \ \ \ \ \  \ \ \ \ \ b_2 \\ 
\underbrace{(1 \  2 \ 3 \ 4) \ (5 \ 6)} & \underbrace{7 \ 8 \ 9
    \ 10 \ 11 \ 12} & \underbrace{ (13 \ 14) \ (15 \ 16 \ 17 \ 18)} \\ 
 \mathrm{small} & \mathrm{intermediate} &
    \mathrm{big} \\
\end{array}
\end{align*}

The permutation $\pi$ is constructed by spreading out the intermediate
elements with $n/2$ and $n/2+1$ remaining fixed points (see
below). The blocks of 
big elements are then inserted in the order $b_{(k-1)/2},\dots,b_1$
in the spaces 
between the smaller intermediate elements while the blocks of small
elements are inserted in the order $s_1,\dots,s_{(k-1)/2}$ in the
spaces between the larger intermediate elements. To obtain $\tau$ we apply the
transposition $(n/2,n/2+1)$ to $\pi$. The
permutations are defined in this manner with a view to decomposing
them into increasing and decreasing sequences of desired cardinalities.

\begin{align*}
\pi = 7 (15 \  16 \ 17 \ 18) \ 8 \ (13 \ 14) \ \underline{9
    \ 10} \ (5 \ 6) \ 11 \ (1 \ 2 \ 3 \ 4) \ 12 \\
\tau = 7 (15 \  16 \ 17 \ 18) \ 8 \ (13 \ 14) \ \underline{10
    \ 9} \ (5 \ 6) \ 11 \ (1 \ 2 \ 3 \ 4) \ 12 \\
\end{align*}

From the construction we claim that $\pi$ and $\tau$ can be
decomposed into a disjoint union of increasing subsequences of cardinalities
$(k+1,k-1,k-1,\dots,2,2)$ and $(k,k,k-2,k-2,\dots,1,1)$
respectively. For $\pi$ the increasing sequences consist of  
(i) The intermediate elements, which in our example is $7,8,9,10,11,12$,
(ii) The blocks of small elements, i.e., $1,2,3,4$ and $5,6$ and (iii) The
blocks of big elements, i.e., $15,16,17,18$ and $ 13,14$.

The permutation $\tau$ can be decomposed into the increasing
subsequences of the following three types:
(i) An intermediate element and the block of big
elements following it, which in the example are $7,15,16,17,18$ and $8,13,14$,
(ii) A block of small elements and the following 
intermediate element, i.e., $5,6,11$ and $1,2,3,4,12$ and (iii) The two
subsequences of length one consisting of one of the two middle
intermediate elements, i.e. $10$ and $9$.

The proof that $\pi$ and $\tau$ have the shapes
$\lambda=(k+1,k-1,k-1,\dots,2,2)$ and 
$\mu=(k,k,k-2,k-2,\dots,1,1)$ respectively uses Lemma
\ref{lem:decreasing-witness}. It is enough to decompose $\pi$ and
$\tau$ into a union of decreasing sequences whose
cardinalities are given by the respective conjugate sequences. 
%Note that the conjugate of the shape $\lambda$ is the
%shape $\mu$ and vice-versa.
Note that as it happens, the shapes $\lambda$ and $\mu$ are conjugates.

We assign the elements of $\pi$ to decreasing subsequences
$d_1,\dots,d_{k+1}$ of cardinalities $k,k,\dots,1,1$.
as follows. Since the $d_i$ are subsequences, elements in them appear in the
same order as in the permutation. Secondly, the assignment is made so
that each subsequence has exactly one of the intermediate elements,
and it appears after any of the big elements and before any of the
small elements. (There is more than one way to do this.) We first see how this is done in the example. 
 
We first construct $d_1$ and $d_2$ which are both decreasing sequences 
of length $5$.  The largest elements in the blocks $b_i$, the
largest elements in the blocks $s_i$ and one of
the middle intermediate elements are assigned to $d_1$. Then we choose
$d_2$ similarly from among the remaining elements.

\begin{align*}
d_1 \ : \  7 (15 \  16 \ 17 \ \framebox{18}) \ 8 \ (13
\ \framebox{14}) \ \framebox{9} \ 10 \ (5 \ \framebox{6}) \ 11 \ (1 \ 2 \ 3
\ \framebox{4}) \ 12 \\
d_2 \ : \  7 (15 \  16 \ \framebox{17} \ \cancel{18}) \ 8 \ (\framebox{13}
\ \cancel{14}) \ \cancel{9} \ \framebox{10} \ (\framebox{5}
\ \cancel{6}) \ 11 \ (1 \ 2 \ \framebox{3} \ \cancel{4}) \ 12 
\end{align*}

The remaining elements can be seen to have the same structure
recursively (the remaining elements appear in the same relative order as would
the elements of the permutation for $n=8$), where the brackets
indicate blocks of big and small elements as before.
\begin{align*}
7 (15 \ 16) \ 8 \ 11 \ (1 \ 2 ) \ 12 
\end{align*}
To assign elements to $d_3$ and $d_4$, we want to continue with the
strategy of choosing the largest elements that remain in the
blocks. Note that since the big 
elements 13 and 14 have been assigned, there are no big elements that follow
the element 8, and it now becomes ``available''.
Thus $d_3$ and $d_4$ are constructed by assigning the largest elements
that remain in the small and big blocks and one of the remaining intermediate
elements in the middle of the blocks. 
\begin{align*}
d_3 \ : \  7 (15 \ \  \framebox{16} \ \cancel{17} \ \cancel{18}\ )
\ \framebox{8} \ (\cancel{13} 
\ \cancel{14}) \ \cancel{9} \ \cancel{10} \ (\cancel{5}
\ \cancel{6}) \ \ 11 \ \ (\ 1 \  \ \framebox{2} \ \cancel{3} \ \cancel{4})
\ 12  \\
d_4 \ : \  7 (\framebox{15} \  \ \cancel{16} \ \  \cancel{17} \
 \cancel{18}) \ \ \cancel{8} \ \ (\cancel{13} 
\ \cancel{14}) \ \cancel{9} \ \cancel{10} \ (\cancel{5}
\ \cancel{6}) \ \framebox{11} \ (\framebox{1} \ \ \cancel{2} \ \ \cancel{3}
\ \cancel{4}) \ 12
\end{align*}
Proceeding the same way, we obtain the subsequences: $d_1=18,14,9,6,4$,
$d_2=17,13,10,5,3$, $d_3=16,8,2$, $d_4=15,11,1$, $d_5=7$, $d_6=12$. In
general, the assignment is done as follows.
\begin{itemize}
\item The $i$-th largest element in each block of big elements, is assigned to
  the subsequence $d_i$.
\item The $i$-th largest element in each block of small elements, is assigned to
  the subsequence $d_i$.
\item For the intermediate elements, assign the lower $(k+1)/2$
  elements to the subsequences $d_{k},d_{k-2},\dots,d_1$ (in that
  order) and the top 
  $(k+1)/2$ elements to $d_2,\dots,d_{k-1},d_{k+1}$ (in that order).
\end{itemize}
Clearly, this is a decomposition of $[n]$ with
exactly $k-2\lfloor (i-1)/2 \rfloor$ elements in $d_i$.
It remains to show that each $d_i$
is a decreasing subsequence. By the construction of the
permutation, the big and small elements in $d_i$ form a decreasing
subsequence since each of them is from a different block. Secondly, the
intermediate element
in $d_i$ appears after all the big elements and before any of the the
small ones.  

Similarly, for the permutation $\tau$, we define the decreasing
subsequences $f_1,\dots,f_k$ of
cardinalities $k+1,k-1,k-1,\dots,2,2$, where $|f_i| = k+1-2\lfloor i/2\rfloor$.
As before, the assignment is made so that each sequence but for one
(which has the two middle intermediate elements) has at most
one intermediate element, and at most one element from each of the
small and the big blocks.
In our example, we construct $f_1$, a subsequence of length $6$, by
taking the largest element from each block and the two middle
intermediate elements.

\begin{align*}
f_1 \ : \  7 (15 \  16 \ 17 \ \framebox{18}) \ 8 \ (13
\ \framebox{14}) \ \framebox{10} \ \framebox{9} \ (5 \ \framebox{6})
\ 11 \ (1 \ 2 \ 3 
\ \framebox{4}) \ 12
\end{align*}

Next, we choose $f_2$ and $f_3$ which are both subsequences of length
$4$. At this point, we cannot continue to follow the
strategy of assigning the largest elements from each block to $f_2$
(by choosing $17,13,5,3$) as in the next step
we would fail to construct $f_3$ of length $4$. Instead, note that
when only one element remains in a block of small elements, the
intermediate element which follows that block has not yet been
assigned and it does not follow any other small elements. Thus the
strategy for $f_2$ is to assign to it the largest elements 
from all blocks except from $s_1$ in which only one element remains,
and to assign the intermediate element following $s_1$ to $f_2$. To
construct $f_3$, we take the largest remaining elements in all
the blocks, and the intermediate element that precedes the block of
big elements whose smallest element was assigned to
$f_2$. Diagrammatically, we have:

 \begin{align*}
f_2 \ : \  7 (15 \  \ 16 \ \ \framebox{17} \ \cancel{18}) \ \ 8 \ \ (\framebox{13}
\ \cancel{14}) \ \cancel{10} \ \cancel{9} \ (\ 5 \ \ \cancel{6})
\ \framebox{11} \ (1 \ \ 2 \ \  \framebox{3} 
\ \cancel{4}) \ 12\\
f_3 \ : \  7 (15 \  \framebox{16} \ \ \cancel{17} \ \ \cancel{18}) \ \framebox{8}
\ (\ \cancel{13} \ \ \cancel{14}) \ \cancel{10} \ \cancel{9} \ (\framebox{5} \ \cancel{6})
\ \ \cancel{11} \ \ (1 \ \framebox{2} \ \ \cancel{3} \
\ \cancel{4}) \ 12
\end{align*}
Repeating the same arguments for the remaining elements, we obtain the following subsequences for the example: $f_1=18,14,10,9,6,4$,
$f_2=17,13,11,3$, $f_3=16,8,5,2$, $f_4=15,12$, $f_5=7,1$. In general,
the subsequences can be defined as follows.
\begin{itemize}
\item The $i$-th largest element in each block of big elements is assigned
  to $f_i$.
\item The smallest element in a block of small elements $s_j$ is assigned
  to $f_{2j+1}$. Among the remaining elements, the $i$-th largest
  element goes to $f_i$.
\item The lower $(k-1)/2$ of the intermediate
  elements go to $f_{k},f_{k-2},\dots,f_{3}$ (in that
  order). The top 
  $(k-1)/2$ elements to $f_2,\dots,f_{k-1}$ (in that order).
  The two middle intermediate elements are in $f_1$.
\end{itemize}
As before, the $f_i$ constitute a decomposition and they have the
appropriate sizes. By construction, the big and
small elements in any subsequence $f_i$ form a decreasing
subsequence. Lastly, for $i \neq 1$ there is at most one
intermediate element in $f_i$ and if one exists, it appears after
all the big elements and before all the small ones. For
$i=1$, the two intermediate elements appear consecutively in
decreasing order, after all big elements and before all small ones.
Thus, $\pi$ and $\tau$ have the claimed shapes and
it follows that
\[
\Delta = \frac{k+1}{2}  = \sqrt\frac n2.
\]

For $n$ not of the form $(k+1)^2/2$, we construct two permutations
as follows. Let $n_0 < n$ be the largest integer such that $n_0 =
(k+1)^2/2$ for odd $k$. The first $n_0$ elements of $\pi$ and $\tau$
are set according to the construction above on $n_0$ elements. The
last $n-n_0$ elements of both $\pi$ and $\tau$ are $n_0+1,\dots,
n$. Then, we have that 
\[
\Delta = \sqrt \frac{ n_0}{2} \geq (1-o(1))\sqrt\frac n2. \qedhere
\]
\end{proof}

We have carried out computer simulations and found other pairs
of permutations for which the bound holds with equality. Several mysteries
remain here, a few of which we mention in Section~\ref{sec:conclusions}.

\section{Bounds on the Lipschitz Constant for $t>1$}
\label{sec:upper-bounds-general-t}

In this section we show bounds on the Lipschitz constant for $t>1$.
Extending the arguments from the previous section for both
the upper and lower bound gives bounds that are tight up to constant
factors for $t=O(1)$. In the latter 
half of this section we give a more complicated argument that
yields an improved upper bound for general $t$.

\subsection{A construction for permutations at linear distance $t$.}
The construction for the case of one transposition can be extended to
the case of more than one transposition as follows. 

\begin{theorem}
\label{thm:construction:t>1}
Let $t \le n/2$. For every $n$ there are permutations $\pi,\tau \in S_n$ with
$d(\pi,\tau)=t$ and respective shapes $\lambda,\mu$ 
such that $\Delta(\lambda,\mu) \geq (1-\sqrt{t/2n})\sqrt{nt/2}$. 
\end{theorem}
\begin{proof}
Let $k = \lfloor \sqrt{2n/t} -1 \rfloor$, and $m=(k+1)^2/2$
so that $mt\le n$. Divide the first $mt$ elements of $[n]$ into $t$
blocks of length $m$ each. To construct the permutations, in each 
block we permute the 
elements as in the construction for one transposition, and then concatenate the
blocks with the remaining $n-mt$ elements following. Then, it is not
difficult to see that the RSK 
algorithm on this pair of permutations will result in a shape with $t$ of the
smaller Young diagrams corresponding to each block being pasted one after the
other, with an additional $n-mt$ boxes in the top row of each
diagram. Then, $\Delta=t \sqrt{m/2}  \geq
\sqrt{nt/2}(1-\sqrt{t/2n})$. Thus when $t=o(n)$, 
$\Delta \ge (1-o(1))\sqrt{nt/2}$. 
\end{proof}

\subsection{Upper Bounds}

We start with an easy observation:

\begin{theorem}
Let $\pi$ and $\tau$ be permutations in $\mathcal S_n$ such that
$d(\pi,\tau)=t$. Let $\lambda$ and $\mu$ be the respective Young diagrams. Then
\[
\Delta= \Delta(\lambda,\mu) \leq t\sqrt{\frac n2}.
\] 
\label{thm:upperbound-easy}
\end{theorem}

\begin{proof}
Since $d(\pi,\tau)=t$, there is a sequence of permutations 
$\pi=\sigma_0,\sigma_1, \dots,\sigma_t=\tau$ such that for each $0
\le i < t$, $\sigma_i$ and $\sigma_{i+1}$ differ by an adjacent
transposition.
The distance $\Delta(\cdot,\cdot)$ is a metric on diagrams and hence the
bound follows by the triangle inequality from Theorem
\ref{thm:upperbound-t=1}. 
\end{proof}
We do not see how to appropriately adapt the bijective argument of Theorem
\ref{thm:upperbound-t=1}. However, the following argument yields a near-optimal
bound.
\begin{theorem}
Let $\pi,\tau \in \mathcal S_n$ be such that $d(\pi,\tau)=t$. Let
$\lambda, \mu$ be the corresponding diagrams. Then
$$\Delta(\lambda,\mu) \leq O(\sqrt{n t \ln t}).$$ 
\label{thm:t-3/4}
\end{theorem}
%However, we do not believe that the bound above is likely to
%be tight and think that the correct bound on $\Delta$ may be $O(\sqrt{nt\ln
%  t})$. 
We start by showing some preliminary results that will be
useful in the proof. Suppose $\pi$ and $\tau$ are two permutations such that
$d(\pi,\tau)=t$. Let $\pi=\sigma_0,\sigma_1, \dots,\\ \sigma_t  =\tau$ be
a sequence of permutations such that for each $0
\le i < t$, $\sigma_i$ and $\sigma_{i+1}$ differ by an adjacent
transposition. Say that $s$ (resp. $r$) of the transpositions put
the relevant pair in decreasing (resp. increasing)
order, where $t=r+s$. Let $\lambda$ and $\mu$ be the
diagrams corresponding to $\pi$ and $\tau$ respectively.
\begin{lemma}
Let $\pi,\tau$ be as above. Then,
\begin{align}
\label{eq:running-sum}
\forall \ 1 \leq j \leq n,  \ \ \ \  \sum_{i=1}^j \mu_i - r\leq
\sum_{i=1}^j \lambda_i \leq \sum_{i=1}^j \mu_i +s
\end{align}
\label{lem:running-sum}
\end{lemma}
\begin{proof}
For each pair $\sigma_i, \sigma_{i+1}$ in
the sequence of permutations
$\pi=\sigma_0,\sigma_1, \dots,\sigma_t=\tau$, by
Proposition \ref{prop:running-sum-t1} the inequality 
\eqref{eq:running-sum-t1} holds for the diagrams corresponding to
  $\sigma_i$ and $\sigma_{i+1}$. The result is obtained by adding up
  all these inequalities. 
\end{proof}
\begin{figure}[t]
\center   
\resizebox{7cm}{!}{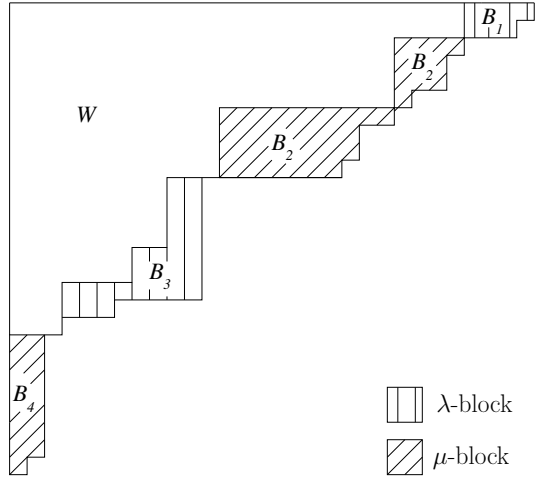}
\caption{The union of $\lambda$ and $\mu$ and the symmetric
  difference split into blocks.} 
\label{fig:tableaux-blocks}
\end{figure}
In Figure \ref{fig:tableaux-blocks} we depict the union of two
diagrams. Their intersection is labeled $W$ as before. We split
the symmetric difference of the two diagrams into {\em blocks}.
We say that $j$ indexes a $\lambda$-row if $\lambda_j > \mu_j$.
A maximal interval of $\lambda$-rows determines
a {\em $\lambda$-pre-block}. A maximal collection
of consecutive $\lambda$-pre-blocks constitutes a
{\em $\lambda$-block}. We likewise define $\mu$-blocks.
Blocks are  labeled $B_i$ as in the figure.
The number of cells in a set $S$ will be denoted by $A(S)$.
We use the following fact about the sizes of the blocks.
\begin{proposition}
\label{prop:blocks-lessthan-t}
Let $d(\pi,\tau)=t$ with corresponding diagrams $\lambda$, $\mu$
and let $B$ be a block in the union of the diagrams, then $A(B) \leq t$. 
\end{proposition}
\begin{proof}
This bound is obtained from Lemma \ref{lem:running-sum} as follows.  
Let $B$ reside in the set of rows $I$ of the diagram. Assuming it
exists, let $i_0$
be the row just preceding $I$, and $i_1=\max I$. Then the bound is obtained
by subtracting the 
inequality \eqref{eq:running-sum} corresponding to $j=i_0$ from the
inequality
corresponding to $j=i_1$, and using the fact that $r+s=t$. If there is
no row $i_0$, then the bound is immediate from the inequality for $j=i_1$. 
\end{proof}
%\begin{figure}
%\center \includegraphics[width=7cm]{tableaux-corners.pdf}
%\caption{An upper bound on $k$ by bijection.}
%\label{fig:tableaux-corners}
%\end{figure}
%Suppose that there are $k$ blocks in the symmetric
%difference. By Proposition \ref{prop:blocks-lessthan-t}
%\[
%\Delta = \frac12 \sum_{i=1}^kA(B_i) \leq  \frac12 kt.
%\]
%Let $(i,j)$ be the leftmost cell of the top row in some block (see Figure
%\ref{fig:tableaux-corners}), and let $(i',j')$ similarly come from
%another block. Then $i \neq i'$ and $j \neq j'$, and
%$(\min(i,i'),\min(j.j')) \in W$. 
%Consequently,
%\[
%\frac{k^2}{2} \leq n-\Delta+\frac{k}{2} \leq n
%\]
%which combined with the above observation implies the bound
%\[
%\Delta \leq t\sqrt{\frac n2}.
%\]
%\subsection{A better upper bound}
The main step in the proof of Theorem \ref{thm:t-3/4} is the
following lemma about two sequences of integers. 
\begin{lemma}
\label{lem:int-seq}
Let $k \geq 2$, $T\geq 3$ and let $a_1,\dots,a_k$ and
$b_1,\dots,b_k$ be two sequences of positive integers.
Denote $\Delta = \sum_{i=1}^k a_ib_i$ and $N=\sum_{1 \leq i \leq j \leq k} a_ib_j.$
If
\[
a_1=b_k=1 \ \mathrm{and} \ \forall i, \ a_i b_i \leq T,
\]
then
\[
\Delta \leq \sqrt{32 N T \ln T}.
\]
This bound is tight up to constants.
\end{lemma}
We first show how to derive the theorem from Lemma \ref{lem:int-seq}.

Let $\lambda$ and $\mu$ be two diagrams of size $n$ (not necessarily
corresponding to permutations at distance $t$). For the union of these
diagrams, define the blocks of
the symmetric difference $\{B_i\}$ and $W$ as before. Suppose that for
each block $B$, $A(B) \leq t$. To prove Theorem \ref{thm:t-3/4}, it
is sufficient to show that for these diagrams,
 \begin{align}
\label{eq:bound-on-diagrams}
\frac{1}{2}\sum_i A(B_i) \leq O\left(\sqrt{t \ln t\left[A(W)+\frac{1}{2}
\sum_i A(B_i)\right]}\right)
\end{align}
With this formulation in mind, we can make the following assumptions
about the pair of diagrams. The aim is to make a number of
transformations and show that the pair of diagrams can be assumed to
be of the form shown
in Figure \ref{fig:diagrams-with-blocks-final}.
\begin{figure}[t]
\center \includegraphics[width=7cm]{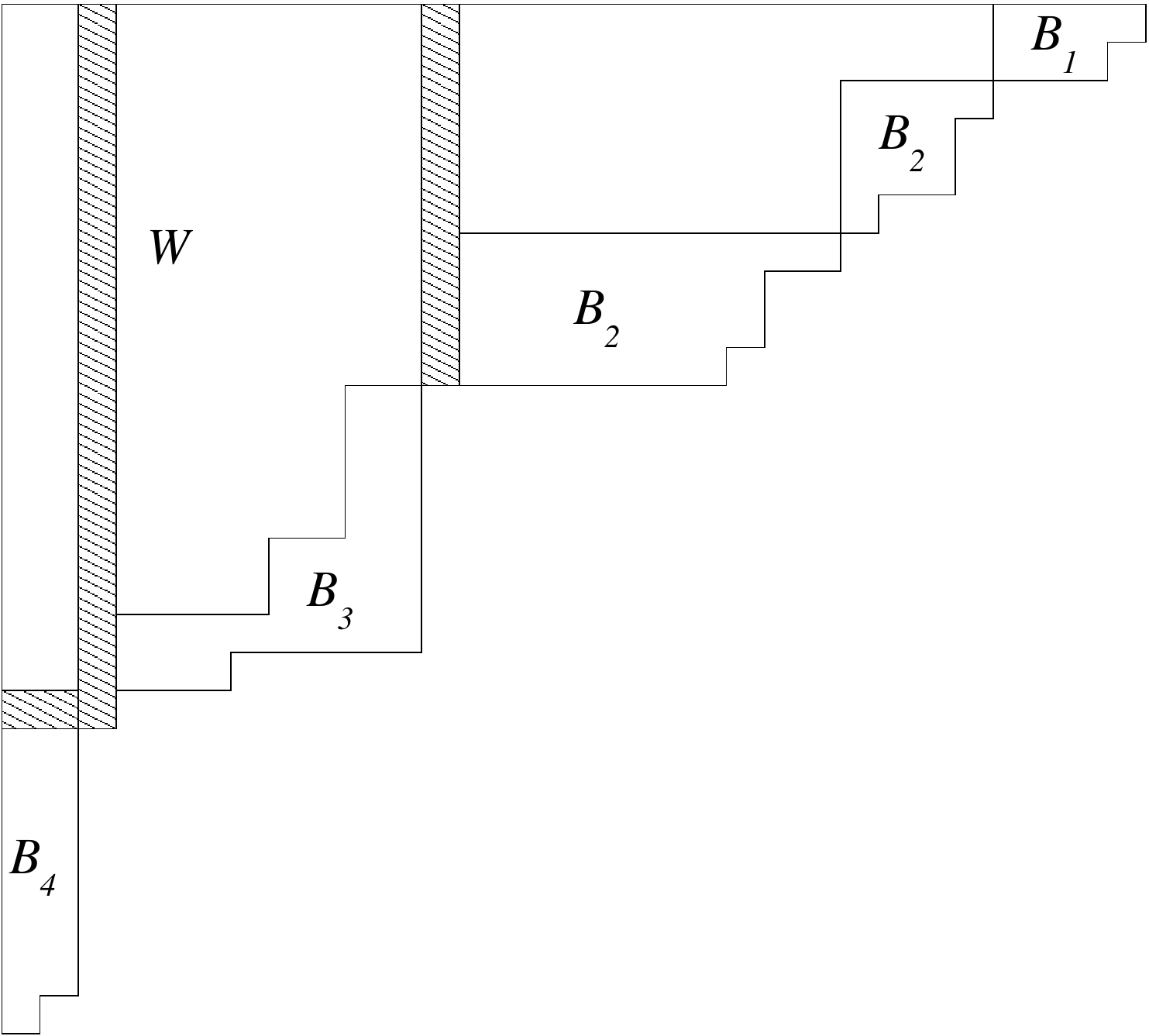}
\caption{Rows and columns of $W$ that may be removed.}
\label{fig:tableaux-shaded-row}
\end{figure}

\begin{reduction}\label{assump:1}

For any row $i$, $\lambda_i \ne \mu_i$ and similarly, for any column
$j$, $\lambda_j' \ne \mu_j'$.
If this is not the case (as in the shaded part of Figure
  \ref{fig:tableaux-shaded-row}), we delete such rows or columns
from both $\lambda$ and $\mu$. Consequently, $A(W)$ decreases, whereas
$\sum_i A(B_i)$ remains unchanged.
Thus, if the bound holds for the new pair of diagrams, it holds
as well for the old pair.  \end{reduction}

\begin{reduction}\label{assump:2}
In general, each block is a skew-diagram (the
set theoretic difference of a diagram and another contained in
it). However, as we show, we may assume it is a Young diagram. 
The dotted lines in Figure \ref{fig:blocks-top-left} mark the
``shade'' of a block in $W$ determined by its top row and leftmost
column. If a block is not a (left-aligned) 
tableau, we can change it to one by removing the cells of $W$ in its
shade and replacing it with a Young diagram of area $A(B)$ contained
in the union of the block and its shade.
\begin{figure}
\center \includegraphics[width=7cm]{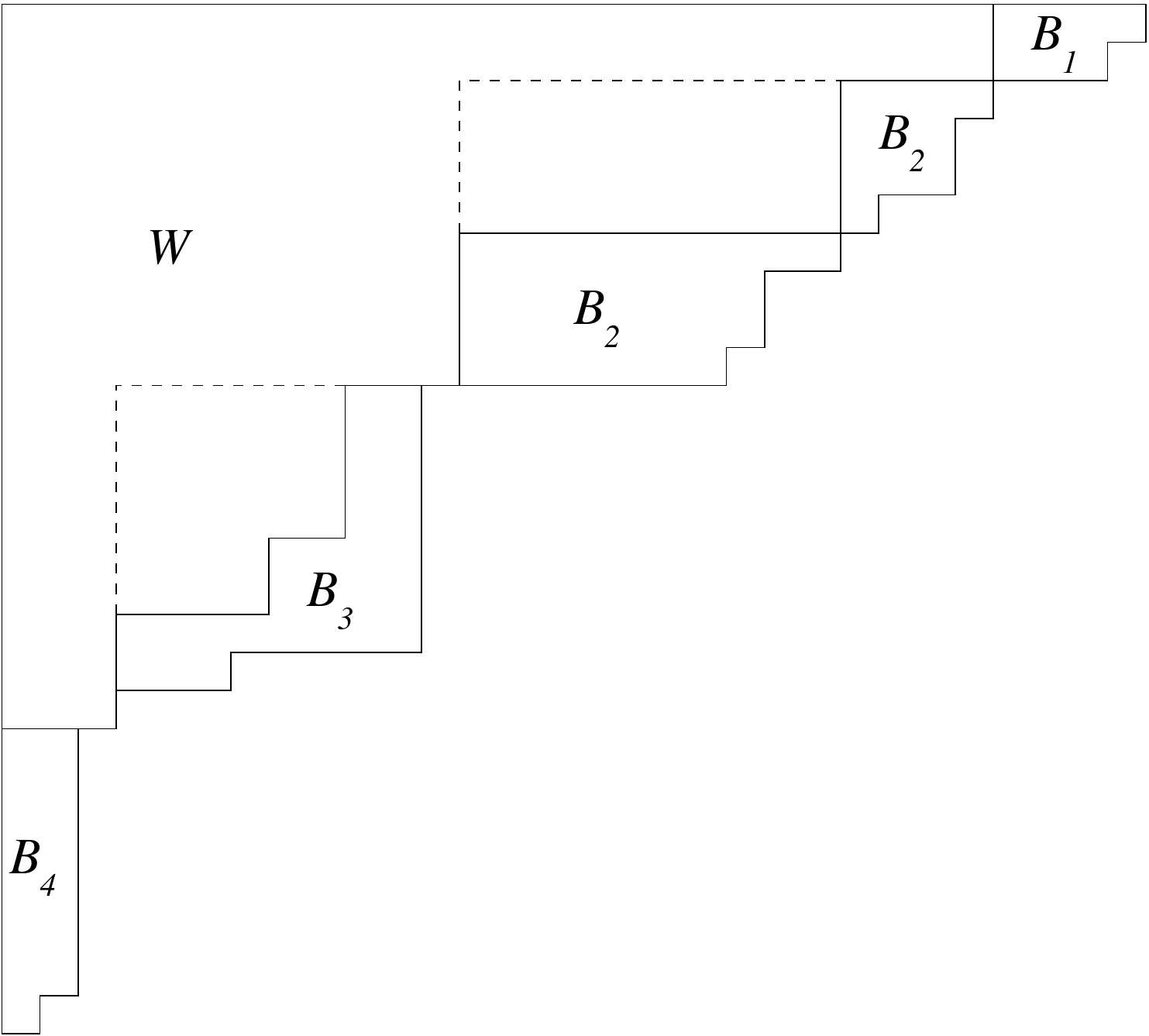}
\caption{The top left corner of the block.}
\label{fig:blocks-top-left}
\end{figure}
\begin{figure}
\center \includegraphics[width=7cm]{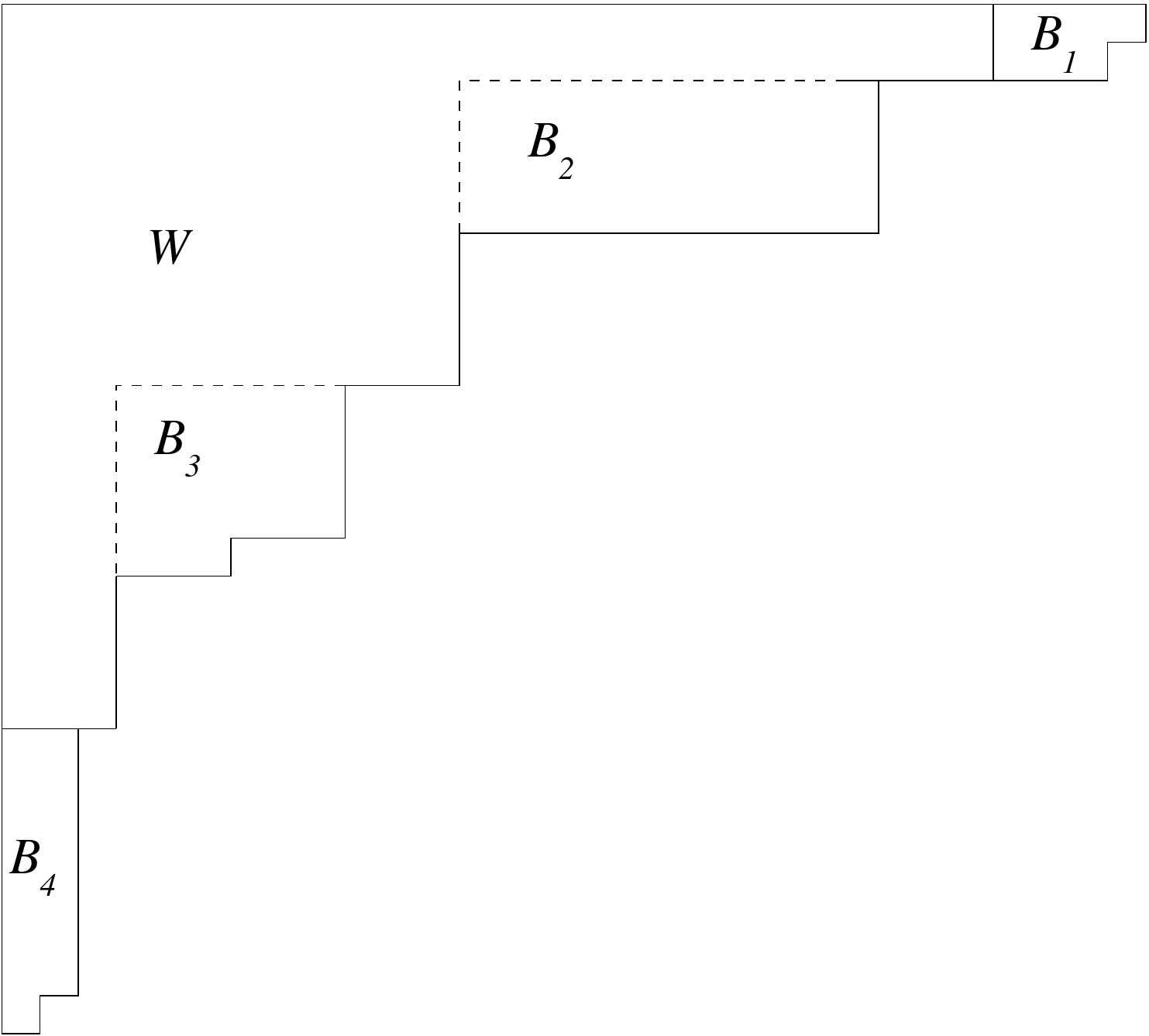}
\caption{All blocks are Young diagrams.}
\label{fig:blocks-top-left-fixed}
\end{figure}
\begin{figure}
\center \includegraphics[width=7cm]{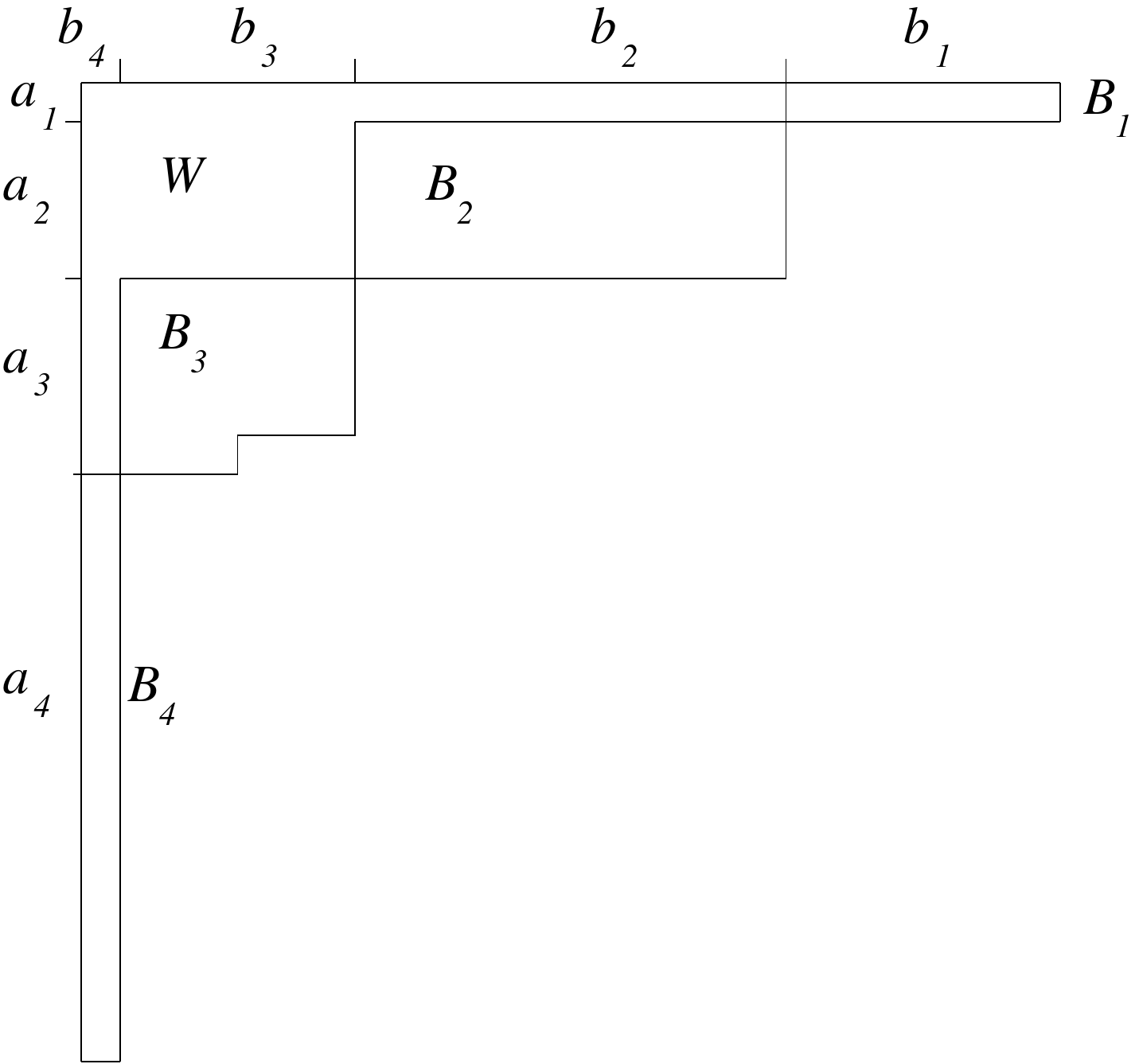}
\caption{A box of side lengths $a_i$ and $b_i$ bounds $B_i$.}
\label{fig:diagrams-with-blocks-final}
\end{figure}
\end{reduction}
This transformation decreases $A(W)$ and keeps the size of the block
fixed.
Secondly, we may assume that the transformation is done so that all rows of a
block, with the possible exception of the last one have the same length.
The result of such a transformation on the blocks $B_2$ and $B_3$ is
shown in Figure \ref{fig:blocks-top-left-fixed}.

%We can transform the diagrams to a new pair satisfying this assumption
%as follows. As before, we will delete cells in $W$ and the remaining
%cells will determine the intersection of the new pair of diagrams. The
%number of cells in a block will stay the same, though their
%co-ordinates will change. Cells of each $\lambda$-block will be assigned to one
%of the new diagrams and the cells of $\mu$-blocks to the other.

%Fix a block $B$. Suppose first that $A(B)$ is less than or equal to the number
%cells of $W$ included  
%between the dotted lines and the block. Then, we can delete these cells
%of $W$, and rearrange the cells of  
%$B$ starting at the top left corner and
%forming a diagram. This transformation keeps the size of the block
%$A(B)$ fixed and decreases $A(W)$.

%In the other case, $A(B)$ is larger than the number cells of $W$ included
%between the dotted lines and $B$.
%Then, we can delete these cells of $W$ and replace them by taking cells from the
%surface of $B$ until we arrive at a diagram. 

\begin{reduction}\label{assump:3} We may assume that the topmost
  block $B_1$ has a single row. Otherwise, we can shift all the cells
  of $B_1$ to the first row without 
  changing any $A(B_i)$ or $A(W)$. We can then delete any rows of $W$
  which are of the same length in $\lambda$ and $\mu$. By similar reasoning,
  we may assume that the bottommost block has a single column.
\end{reduction}

Thus, we may assume that the diagrams are as shown in Figure
\ref{fig:diagrams-with-blocks-final} and that the sizes of the blocks 
are bounded by $t$. As in the figure, let $a_i$ and
$b_i$ denote the lengths of the vertical and horizontal sides of the
rectangle which bounds the block $B_i$. Thus the area $A(W)$ can be
written as a sum of areas of rectangles $a_ib_j$ whose sides are
determined by the side lengths of pairs of blocks. 
Also note that by our construction of the blocks, $a_ib_i \leq 2t$.

To obtain the formulation of the lemma, suppose that we add cells to
the last row of each block to  
  complete it to a rectangle. Denote the modified blocks by
  $B'$. Then for each block, $A(B') \leq 2A(B)$. If we show the bound for
  these modified diagrams with a bound of $2t$ for each block,
  then the bound is implied for the original diagrams since the constants
  can be absorbed by the $O(\cdot)$. Formally, this follows
  from the following inequalities. 

\begin{enumerate}
\item $\frac{1}{2}\sum_i A(B_i) \le \frac{1}{2}\sum_i A(B_i') $
\item $A(W)+\frac{1}{2}\sum_i A(B_i') \le 2 \left(A(W)+\frac{1}{2}\sum_i A(B_i) \right)$
\end{enumerate}

Thus Lemma \ref{lem:int-seq} implies the bound
\eqref{eq:bound-on-diagrams} for a pair of diagrams as above and we
have verified that to prove Theorem \ref{thm:t-3/4} it is 
sufficient to prove the lemma. \\

\noindent{\bf Proof of Lemma \ref{lem:int-seq}:} We will minimize $N/\Delta^2$. 
For $k=2$, the lemma can 
  be easily verified by calculation once we use the fact that
  $a_1=b_2=1$. Thus we 
  will assume that $k \geq 3$. 
Consider the following relaxation of the
minimization problem where the $a_i,b_j$ are not necessarily integral.

\begin{align*}
\nonumber \min \frac{N}{\Delta^2}=\frac{\displaystyle\sum_{1 \leq i
    \leq j \leq k} a_ib_j}  
{\left(\displaystyle\sum_{i=1}^k a_ib_i \right)^2}\\ 
s.t.   \ \ \ \ \  \ \ \ \ \ \ \ a_1= b_k=1\\
a_ib_i \le T, \ 1 \le i \le k \\
a_i \geq 1, \ 2 \leq i \leq k \\
b_i \geq 1, \ 1\leq i \leq k-1
\end{align*}

We will use the method of Lagrange multipliers (see Appendix \ref{app:lagrange}
for a brief introduction) to obtain a lower bound on the value of the
objective above at any local optimum. Since the problem is a
relaxation of the discrete minimization problem, this also lower
bounds the objective of the discrete problem.
We obtain the following Lagrangian for the relaxation above.
\begin{align*}
\min \mathcal L  = & \frac{N} 
{\Delta^2} - \sum_{i=1}^k\lambda_i(a_ib_i-T)-
\sum_{i=1}^k \mu_i(a_i - 1)-
\sum_{i=1}^{k} \nu_i(b_i - 1) \nonumber \\
\end{align*}
The Karush-Kuhn-Tucker conditions yield the
following necessary conditions for minimality. 
\begin{align}
& \frac{\partial}{\partial a_i}\mathcal L = \frac{\partial}{\partial
  a_i}\frac{N}{\Delta^2} - \lambda_ib_i - \mu_i = 0,  \quad & 1 \leq i \leq k
  \label{eq:del-a} \\
& \frac{\partial}{\partial b_i}\mathcal L =
\frac{\partial}{\partial  
  b_i}\frac{N}{\Delta^2} - \lambda_ia_i - \nu_i = 0,  \quad & 1 \leq i \leq
  k \label{eq:del-b} \\
& \lambda_i \ge 0,~~ \lambda_i(T-a_ib_i)=0, \quad & 1 \le i \le k \nonumber\\
& \mu_i \geq 0, \quad \mu_i(a_i-1) = 0, \quad & 2 \leq
i \leq k \nonumber\\ 
& \nu_i \geq 0, \quad \nu_i(b_i-1) = 0,
\quad & 1\leq i \leq k-1 \label{eq:Lagrange-zeros} 
\end{align}

From these conditions, we can show that at optimality either $a_ib_i
=T$ or $1$. Suppose  
that for some $i$, $a_ib_i<T$. Note that by the conditions above, this
implies $\lambda_i=0$. Now, if $a_ib_i \neq 1$, at least one of $a_i$
or $b_i$ is $>1$. Assume without loss of generality that $b_i >1$ (the
argument in the other case is exactly the same). In this case
$\nu_i=0$ by \eqref{eq:Lagrange-zeros} . Hence from \eqref{eq:del-b}
above, we have 
\begin{align*}
\frac{\partial}{\partial b_i}\frac{N}{\Delta^2}=0
\end{align*}
and therefore, since $\Delta>0$
\begin{align}
\nonumber & \Delta^2\frac{\partial N}{\partial b_i} = 2\Delta N\frac{\partial
  \Delta}{\partial b_i}\\
\Rightarrow \quad & \frac{\displaystyle\sum_{j=1}^i a_j}{2a_i} =
\frac{N}{\Delta} \label{eq:optimality-N-Delta}
\end{align} 
Now we show that it is possible to increase $b_i$ by a factor
$(1+\varepsilon)$ for $\varepsilon>0$ so that $N/\Delta^2$ decreases
and we can conclude that the solution is not optimal.
This is allowed, at least for $\varepsilon>0$ small enough,
since, by assumption $a_i b_i < T$. Let $N'$ and
$\Delta'$ be the summations as defined before for the sequences where we
replace $b_i$ by $b_i(1+\varepsilon)$.
\begin{align*}
\frac{N'}{(\Delta')^2} = \frac{N+ b_i \varepsilon
  \displaystyle\sum_{j=1}^i a_j}{(\Delta + a_ib_i\varepsilon)^2}
\end{align*}
To prove the claim $N'/(\Delta')^2< N/\Delta^2$, using the right hand side
above, it is enough to show that
\begin{align*}
\Delta^2 b_i\varepsilon\displaystyle\sum_{j=1}^i a_j  <
N\left(2\Delta a_ib_i\varepsilon + (a_ib_i\varepsilon)^2\right)
\end{align*}
Or equivalently, dividing throughout by $2 a_i b_i \varepsilon
\Delta^2$, that
\begin{align*}
\frac{\displaystyle\sum_{j=1}^i a_j}{2a_i}  < \frac{N}{\Delta} +
\frac{Na_ib_i \varepsilon}{2\Delta^2}.
\end{align*}
This inequality follows by \eqref{eq:optimality-N-Delta}. The left-hand
term equals the first term on the right and $\varepsilon>0$.

The next step is to argue that it is enough to show the claimed bound
assuming that the blocks
are arranged in a specific manner (i.e., the sequences are of a
certain form). In particular, the blocks of
area $T$ are arranged such that $a_i$ is increasing and $b_i$ is
decreasing. Secondly, the
blocks of area $1$ occur after all blocks such that $a_i \leq b_i$ and
before all blocks such that $a_i> b_i$. This can be argued by
noticing that such an arrangement can be achieved by exchanging blocks
which are out of order since $\Delta$ remains unchanged and $N$ does not
increase. Thus a lower bound on $N/\Delta^2$ for the modified sequence
is a lower bound on the corresponding quantity for the original sequence.

We next argue that, in fact, w.l.o.g. no block has area $1$.
Recall that we wish to show \[32T \ln T N - \Delta^2
\geq 0.\] We will show that if we add a single block of area $1$ then
 
\begin{align}
32T \ln T N' - \Delta'^2
\ge 32T \ln T N - \Delta^2
\label{eq:no-1-blocks}
\end{align}
where $N'$ and $\Delta'$ are the modified values of $N$ and
$\Delta$. This inequality above allows us to reduce the argument to the
case when there are no blocks of area $1$. Let the shorter sequence have $k$
terms. Note that $\Delta'=\Delta+1$ and the change in the number of cells
$N' - N$ is at least
$\sum_{i=1}^k \min(a_i,b_i)$.

Recall that $k,T \geq 3$ and for $1 \le i \le k$, $a_i,b_i \ge 1$. Thus,
\begin{align*}
(\Delta')^2 - \Delta^2 = & \quad 2\Delta+1 \\ 
 \le & \quad  2kT+1 \\
 \le  & \quad 32k T \ln T \\
\le  & \quad 32 T \ln T \sum_{i=1}^k \min(a_i,b_i)\\
=  & \quad 32 T \ln T (N'-N)
\end{align*}
which implies the required inequality \eqref{eq:no-1-blocks}.

In the next step, we will make a further simplification to the
picture. To summarize, we now know that we may optimize over sequences
such that each block has size $T$, $a_1=b_k=1$, the sequence $a_i$ is
non-decreasing and $b_i$ is non-increasing. The claim is that the
optimal solution is of the form where there is some $i$ such that
$a_i>1$ and $b_i>1$. If not, then it can be checked that $\Delta =
\sqrt{8NT}$ and the claimed bound holds.

We relabel the sequences $a_{-\ell_1},\dots,a_{-1},a_1,\dots,a_{k},
a_{k+1},\dots,a_{k+\ell_2}$ 
and
$b_{-\ell_1},\dots ,\\ b_{-1},b_1,\dots,b_k,b_{k+1},\dots,b_{k+\ell_2}$
where $\ell_1,\ell_2 \ge 1$ and $k \geq 0$ so that $a_i,b_i>1$ for
$1\le i \le k$. Let $N$ and $\Delta$ be 
the corresponding summations as defined before. We can reformulate the
minimization problem as follows. 
\begin{align}
\nonumber \min \frac{N}{\Delta^2}\\
\nonumber s.t.   \ \ \ \ \ \ \ a_i=1=b_{k+j}, a_{k+j}=T=b_i, & \quad -\ell_1
\le i \le -1, \quad 1 \le j \le \ell_2\\
\nonumber a_ib_i = T, & \quad \forall i \\
a_i,b_i \geq 1 & \quad 1 \le i \le k 
\label{eq:min-ab-is-T}
\end{align}

Solving this optimization problem gives the following conditions for
the solutions (see Proposition \ref{prop:geometric} in the Appendix
\ref{sec:minimization} for the detailed calculations). 
The sequences ${b_{1 \leq i \leq k}}$ and (hence ${a_{1 \leq i \leq
    k}}$) are a geometric series with 
\[
c= \frac{b_1}{b_2} = \dots = \frac{b_{k-1}}{b_{k}}
\]
and
\[
c= \frac{a_2}{a_1} = \dots = \frac{a_{k}}{a_{k-1}}
\]
and the ratio between successive terms $c>1$.
Also,
\[
b_k = (c-1)\ell_2.
\]
and
\[
a_1 = (c-1)\ell_1.
\]
Substituting, we also have that
\begin{align}
b_1 = c^{k-1}b_k = c^{k-1}(c-1)\ell_2 \nonumber \\
T = a_1b_1 = c^{k-1}(c-1)^2\ell_1 \ell_2 \label{eq:T-and-b1}
\end{align}
Since $b_k, a_1 > 1$ we have
\begin{align*}
c > 1+ \frac{1}{\max\{\ell_1,\ell_2\}}.
%\label{eq:c-lower-bound}
\end{align*}
Furthermore, $c^{k-1} < T$ and therefore
\begin{align}
k-1 < \frac{\ln T}{\ln c} \leq \frac{\ln T}{\ln
  (1+1/\max\{\ell_1,\ell_2\})}.
\label{eq:k-upper-bound}
\end{align}

In the next step we will show the w.l.o.g. we may assume $\ell_1=\ell_2=1$.
Let $N_0 = \sum_{1 \le i \le j \le k} a_ib_j +\sum_i b_i + \sum_i a_i +1 +2T$
and $\Delta_0 = \sum_{i=1}^k  a_ib_i +2T$. These are the values of the
summations with the first $\ell_1-1$ and last $\ell_2-1$ members of
the sequences removed. Then 
\begin{align*}
N & = N_0 + T\frac{\ell_1(\ell_1-1)}{2}+T\frac{\ell_2(\ell_2-1)}{2} \\
& ~~~~~~~~~~~~ +(\ell_1-1)\sum_{i=1}^k b_i + (\ell_2-1)\sum_{i=1}^k a_i +
(\ell_1-1)(\ell_2-1), \\  
\Delta & = \Delta_0 + T(\ell_1+\ell_2 -2)
\end{align*}

We have above that $\ell_1,\ell_2 \ge 1$. We
will show that the optimal of $32T\ln TN - \Delta^2$ when $N_0$ and
$\Delta_0$ are fixed is at $\ell_1=\ell_2=1$ by showing 
\[
32T\ln TN - \Delta^2 \ge 32T\ln TN_0 - \Delta_0^2.
\]
Wlog, suppose that $\ell_1 \geq \ell_2$ so that $\ell_1 \ge
2$. Therefore by \eqref{eq:k-upper-bound} and using the fact that for
$x \le 1/2$, $\ln(1+x) \ge x/2$, we have
\begin{align}k-1 < \frac{\ln T}{\ln
  (1+1/\max\{\ell_1,\ell_2\})} \leq 2 \ln T \max\{\ell_1,\ell_2\}.
\label{eq:k-upper-lnT}
\end{align}
Now, we have
\begin{align*}
32 T\ln T (N-N_0) \geq  32 T\ln T
  (T\frac{\ell_1(\ell_1-1)}{2}).
\end{align*}
On the other hand, by the bound from \eqref{eq:k-upper-lnT} on $k$,
\begin{align*}
\Delta^2 - \Delta_0^2 & = 2\Delta_0T(\ell_1+\ell_2-2) +
T^2(\ell_1+\ell_2-2)^2\\
& = 2kT^2(\ell_1+\ell_2-2) + T^2(\ell_1+\ell_2-2)^2\\
& \le 4kT^2(\ell_1-1) + 4T^2(\ell_1-1)^2 \\
& \le 12T^2 \ln T\ell_1(\ell_1-1) + 4T^2(\ell_1-1)^2\\
& \le 16T^2 \ln T\ell_1(\ell_1-1)\\
& \le 32 T\ln T
  (T\frac{\ell_1(\ell_1-1)}{2})\\
& \le 32 T\ln T (N-N_0) 
\end{align*}
as required.
Finally, if $\ell_1=\ell_2=1$, then using the fact that $b_1 \leq T$, from
\eqref{eq:T-and-b1}, we obtain that $c>2$. Hence by
\eqref{eq:k-upper-lnT} $k<2\ln T +1$.  Note that if
$\ell_1=\ell_2=1$, $\Delta = (k+2)T$. We can then use the following
straightforward bound. 
\begin{align*}
\frac{N}{\Delta^2} \geq \frac{1}{\Delta} = \frac{1}{(k+2)T} \geq
\frac{1}{4T \ln T}.
\end{align*}

Theorem \ref{thm:t-3/4} now follows. \myqed

As we show next, the upper bound of Lemma \ref{lem:int-seq} is
tight. To construct a pair of 
diagrams where the $a_i$ and $b_i$ are integers
and $\Delta = \Omega(\sqrt{N T \ln T})$ we argue as follows. Let
$T=2^k$, $c=2$, $b_1=1$ and $b_k=T$. Then $\Delta$ is at least
$\Omega(\sqrt{N T \ln T})$. However, the
tightness of Theorem \ref{thm:t-3/4} does not follow from this since
it is not clear that there exist corresponding permutations.

\section{Conclusions}
\label{sec:conclusions}
A number of interesting directions remain for further research.\\

{\bf Characterize extremal permutations.} The permutations constructed
in Section \ref{sec:t=1} achieve the 
maximum difference in the shapes for one transposition. There it
was possible to construct the examples by carefully arranging
the increasing and decreasing sequences. On the other hand,
with the help of a computer, we
observed several other examples whose structure we
do not completely understand. We know that for $\Delta$ to achieve the
upper bound, by Greene's Theorem, the permutations must be
decomposable into unions of increasing sequences whose sizes are given by the
required shape of the diagram. An example of such a pair of
permutations from simulation for $n=18$ is: 
\begin{align*}
 13 \ 14 \ 10 \ 15 \ 6 \ 1 \ 18 \ 2 \ 16 \ \underline{9
    \ 11}  \ 12 \ 3 \ 7 \ 17 \ 8 \ 4 \ 5\\
13 \ 14 \ 10 \ 15 \ 6 \ 1 \ 18 \ 2 \ 16 \ \underline{11
    \ 9} \ 12 \ 3 \ 7 \ 17 \ 8 \ 4 \ 5\\
\end{align*}
Notice that in this example the permutations cannot be
decomposed into conjugate increasing and decreasing sequences
as done in our construction. In our view the class of such
permutations is an intriguing mathematical object. We would
like to know how many such permutations exist, what their
structural properties are etc. This seems like a good subject
for further work in this area.
\vspace{0.1in}

{\bf Constructions for $t>1$ transpositions.} As mentioned, we do not
know whether there exists a pair of permutations corresponding to the
diagrams which are tight for Lemma \ref{lem:int-seq}. We do not see
how to extend our construction for one transposition to this
case. 

Secondly, our constructions achieve $\Omega(\sqrt{nt/2})$
differences when $t \le n/2$. The behavior for larger $t$ is still
unclear. For example, the maximum possible value of $\Delta$ is $n-1$,
and this is uniquely achieved with $t={n \choose 2}$ transpositions.
We also know from Theorem~\ref{thm:construction:t>1} that we can
make $\Delta \ge \Omega(n)$ with $t\le O(n)$. We still do not know
how large $t$ should be to make
$\Delta \ge \alpha n$ with $\alpha$ close to $1$.
\vspace{0.1in}

{\bf Dependence on transpositions.} It would be interesting to
obtain more detailed information about
the change in $\Delta$ as a result of left-multiplication with a
transposition. Knuth and Knuth-dual equivalence
classes characterize transpositions which keep $\Delta$ fixed. What is
the expected change in $\Delta$ for a transposition in a random
permutation? How do the position of the transposition or properties of
the permutation affect the change?\vspace{0.1in}

{\bf Other metrics.} In this work we studied the adjacent transposition metric
on permutations but there are a number of natural measures for the
distance between two permutations which may be worth studying in this
setting. 

For example it can be verified that up to constants, the same bounds
on the Lipschitz constant hold for the distance $d'$ on permutations
with respect to general (not 
necessarily adjacent) transpositions. The lower bounds from Theorems
\ref{thm:lower-bound-one-transp} and \ref{thm:construction:t>1} hold
since the constructions give permutations $\pi$ and $\tau$ which
differ by $1$ and $t$ transpositions respectively. On the other hand,
the upper bounds on 
the Lipschitz constant follow (and are within a constant factor of the
bounds for adjacent transpositions) since by Greene's
Theorem the bounds in Proposition \ref{prop:running-sum-t1} change
only by a small additive constant and in 
Lemma \ref{lem:running-sum} this translates to the absolute value of
the difference between the sums being bounded by $2t$ if the
permutations differ by the multiplication of $t$ transpositions.

\appendix

\section{The Method of Lagrange Multipliers}
\label{app:lagrange}
The method of Lagrange multipliers is used to solve for
the maxima or minima of a real-valued multivariate function subject to
equality constraints. In 
particular, the method gives {\em necessary conditions} for optimality
which are the analog of the conditions on the 
gradient for unconstrained problems. The Karush-Kuhn-Tucker (KKT)
conditions for optimality generalize these to the case
when some of the constraints may be inequalities.
Consider the following optimization problem, where
$\mathbf{x}=(x_1,\dots,x_n)$ and $\alpha_j,\beta_k \in \mathbb{R}$:
\begin{align*}
& \min  f(\mathbf{x})\\
s.t. \ \ & g_j(\mathbf{x}) \le \alpha_j, \ \ j = 1,\dots,\ell\\
& h_k(\mathbf{x}) = \beta_k \ \ k= 1,\dots, m.
\end{align*}
The {\em Lagrangian} for this problem is defined to be the function:
\[
\mathcal{L} := \mathcal{L}(\mathbf{x},\lambda_j,\mu_k) = f(\mathbf{x})
- \displaystyle\sum_{j=1}^\ell \lambda_j ( g_j(\mathbf{x}) - \alpha_j) -
\displaystyle\sum_{k=1}^m \mu_k ( h_k(\mathbf{x}) - \beta_k). 
\]
The KKT conditions say that if a local optimizer $\mathbf{x}^*$
satisfies certain technical ``constraint qualifications'' (explained
below) then there are constants $\lambda_j^*$ ($j=1,\dots,\ell$) and
$\mu_k^*$ ($k=1,\dots,m$) satisfying
\begin{align*}
\nabla{\mathcal{L}} = \nabla{f(\mathbf{x}^*)}
- \displaystyle\sum_{j=1}^\ell \lambda_j^* \nabla{g_j(\mathbf{x}^*)} &
- \displaystyle\sum_{k=1}^m \mu_k^* \nabla{h_k(\mathbf{x}^*)} = 0 \\
\lambda_i^* \ge 0,   & \ \ j = 1,\dots,\ell\\
\lambda_i^*g_i(\mathbf{x}^*) = 0,   & \ \ j = 1,\dots,\ell.\\
\end{align*}
A number of constraint qualifications are known to be sufficient for the
result and in our case, the so-called {\em Mangasarian-Fromovitz}
  constraint qualification holds. This condition requires that at
$\mathbf{x}^*$, the gradients of any active inequality constraints and
the gradients of the equality constraints are {\em positively-linearly
  independent}. A collection of vectors $(v_1,\dots,v_d)$ is
positively-linearly dependent if there are $a_1\ge 0,\dots,a_d \ge 0$,
not all $0$ such that $\sum_i a_iv_i = 0$. for the optimization
problems we consider, the constraint
qualification can be verified without much difficulty for any possible set
of active constraints, so we leave this to the reader
and assume that the KKT conditions are satisfied. For more details
regarding the method of Lagrange multipliers and extensions, the
reader may refer to \cite{Ber}.

\section{Solution to the Minimization Problem}
\label{sec:minimization}
We show below the calculations that solve the minimization problem in
\eqref{eq:min-ab-is-T} which is reproduced below. 
\begin{align*}
\min \frac{N}{\Delta^2}\\
s.t.   \ \ \ \ \ \ \ a_i=1=b_{k+j}, a_{k+j}=T=b_i, & \quad -\ell_1
\le i \le -1, \quad 1 \le j \le \ell_2\\
a_ib_i = T, & \quad \forall i \\
a_i,b_i \geq 1 & \quad 1 \le i \le k 
\end{align*}

\begin{proposition}
At a minimum of the optimization, the sequences ${b_{1 \leq i \leq
    k}}$ and (hence ${a_{1 \leq i \leq k}}$) are a geometric series with 
\[
c= \frac{b_1}{b_2} = \dots = \frac{b_{k-1}}{b_{k}}
\]
and
\[
c= \frac{a_2}{a_1} = \dots = \frac{a_{k}}{a_{k-1}}
\]
and the ratio between successive terms $c>1$.
Also,
\[
b_k = (c-1)\ell_2.
\]
and
\[
a_1 = (c-1)\ell_1.
\]
\label{prop:geometric}
\end{proposition}
\begin{proof}
We obtain the following Lagrangian.
\begin{align*}
\min \mathcal L  = & \frac{N} 
{\Delta^2} - \sum_{i}\lambda_i(a_ib_i-T)-
\sum_{i=1}^k \mu_i(a_i - 1)-
\sum_{i=1}^{k} \nu_i(b_i - 1) \nonumber \\
\end{align*}

From the KKT conditions for optimality, we obtain:

\begin{align}
& \frac{\partial}{\partial a_i}\mathcal L = \frac{\partial}{\partial
  a_i}\frac{N}{\Delta^2} - \lambda_ib_i - \mu_i = 0,  \quad & 1 \leq i \leq k
  \label{eq:simp-del-a}\\
& \frac{\partial}{\partial b_i}\mathcal L = \frac{\partial}{\partial 
  b_i}\frac{N}{\Delta^2} - \lambda_ia_i - \nu_i = 0,  \quad & 1 \leq i \leq
  k \label{eq:simp-del-b}\\
\nonumber
& \mu_i \geq 0, \quad \mu_i(a_i-1) = 0, \quad & 1 \leq
i \leq k \\ 
\nonumber
& \nu_i \geq 0, \quad \nu_i(b_i-1) = 0, \quad & 1\leq i \leq k
\end{align}

As outlined before, we may assume that the optimal solution is such
that $a_i,b_i >1$ for $1 \le i \le k$. Hence by the conditions above,
$\mu_i=\nu_i=0$. 
Performing the differentiations in \eqref{eq:simp-del-a}
(w.r.t. $a_i$) and  \eqref{eq:simp-del-b} (w.r.t. $b_i$) and
multiplying them by
$a_i$ and $b_i$ respectively we obtain the following relations.

\begin{equation}
\frac{a_ib_i N}{ \Delta^4} -
\lambda_i a_ib_i -
\frac{a_i(b_i+\dots+b_k+b_{k+1}+\dots+b_{k+\ell_2})}{\Delta^4} = 0, 
\quad 1 \leq i \leq k
\label{eq:Lagrange-equality-1}
\end{equation}
\begin{equation}
\frac{a_ib_i N}{\Delta^4} -
\lambda_i a_ib_i - \frac{b_i(a_{-\ell_1}+\dots+a_{-1}+a_1+\dots + a_i)}{\Delta^4} = 0, 
\quad 1 \leq i \leq k
\label{eq:Lagrange-equality-2}
\end{equation}

Equating \eqref{eq:Lagrange-equality-1} and
\eqref{eq:Lagrange-equality-2} and cancelling terms, we conclude that since 
$\Delta \neq 0$,

\begin{align}
a_i(b_i+\dots+b_k + \ell_2) = b_i(\ell_1+a_1+\dots+a_i), \quad  1 \leq i \leq k
\label{eq:general-equality}
\end{align}

We can solve the above set of relations as
follows. Dividing \eqref{eq:general-equality} by $a_i$ and using the
equations corresponding to $i$ and $i+1$ there, and that $a_ib_i=T$,
after rearranging terms we obtain the following relations:
\begin{align*}
b_i+ b_{i+1}+\dots+b_k  +\ell_2 & = \frac{b_i}{a_i}(\ell_1+a_1+\dots+a_i) =
b_i^2\left(\ell_1+\frac{1}{b_1}+\dots \frac{1}{b_i}\right)\\
b_{i+1}+\dots+b_k  +\ell_2 & = \frac{b_{i+1}}{a_{i+1}}(\ell_1+a_1+\dots+a_{i+1}) =
b_{i+1}^2\left(\ell_1+\frac{1}{b_1}+\dots+ \frac{1}{b_{i+1}}\right)
\end{align*}
Subtracting we obtain
\begin{align*}
b_i & =   b_i^2\left(\ell_1+\frac{1}{b_1}+\dots
\frac{1}{b_i}\right) - b_{i+1}^2\left(\ell_1+\frac{1}{b_1}+\dots+
\frac{1}{b_{i+1}}\right)  \\
& = ( b_i^2  - b_{i+1}^2)\left(\ell_1+\frac{1}{b_1}+\dots
\frac{1}{b_i}\right) - b_{i+1}, \quad 1 \leq i \leq k
\end{align*}
Rearranging, 
\begin{align}
b_i\left(\ell_1+\frac{1}{b_1}+\dots
\frac{1}{b_{i}}\right) =  b_{i+1}\left(\ell_1+\frac{1}{b_1}+\dots+
\frac{1}{b_{i+1}}\right), \quad 1 \leq i \leq k
\label{eq:b-recurrence}
\end{align}
Let \[H_i = \ell_1+\frac{1}{b_1}+\dots \frac{1}{b_{i}}\] so that
\[ \frac{1}{b_i} = H_i - H_{i-1}.\]
Rearranging \eqref{eq:b-recurrence} and manipulating both sides, we have
\begin{align*}
& \frac{H_i}{b_{i+1}} = \frac{H_{i+1}}{b_i} \\
\Rightarrow \quad & H_i(H_{i+1}-H_i) = H_{i+1}(H_i - H_{i-1})\\
\Rightarrow \quad & \frac{H_i}{H_{i-1}} = \frac{H_{i+1}}{H_i} \\
\Rightarrow \quad & \frac{b_{i-1}}{b_i} = \frac{b_{i}}{b_{i+1}}, \quad
1 < i <  k. 
\end{align*}
In other words, we can conclude that ${b_{1 \leq i \leq k}}$ (and
hence ${a_{1 \leq i \leq k}}$) is a
geometric series.
Let 
\[
c= \frac{b_1}{b_2} = \dots = \frac{b_{k-1}}{b_{k}}
\]
and
\[
c= \frac{a_2}{a_1} = \dots = \frac{a_{k}}{a_{k-1}}.
\]
It can be checked that $c > 1$ since for $c=1$
\eqref{eq:b-recurrence} is not satisfied.

Next, suppose we multiply the equation \eqref{eq:general-equality} by
$c-1$, we obtain

\begin{align*}
(c-1)a_i(b_i+\dots+b_k+\ell_2) & = (c-1)b_i(\ell_1+a_1+\dots+a_i) \\
\Rightarrow \quad \quad \quad a_i(cb_i - b_k +(c-1)\ell_2) & = b_i(ca_i - a_1
+(c-1)\ell_1) \\
\Rightarrow \quad \ \quad \quad \quad \frac{a_i}{b_i}(- b_k +(c-1)\ell_2) & = - a_1
+(c-1)\ell_1 \\ 
\end{align*}
Since the right hand side of the last equality is the same for all
$1 \le i \le k$, from $i=1$ and $i=k$, we obtain
\begin{align*}
 \left(\frac{a_1}{b_1}-\frac{a_k}{b_k}\right)(- b_k +(c-1)\ell_2) = 0
\end{align*}

Now since $c>1$, $\frac{a_1}{b_1} \ne \frac{a_k}{b_k}$ and therefore
\[
b_k = (c-1)\ell_2.
\]

By similar arguments,
\[
a_1 = (c-1)\ell_1. \qedhere
\]
\end{proof}
\end{document}